\documentclass[a4paper,10pt,leqno, english]{amsart}
\title[Equivariant K-Theory of  Central  extensions]
{Equivariant  $K$-Theory  of  central  extensions  and  twisted  equivariant  $K$-Theory: $\sldrei$ and  $\stein$. }

\author{No\'{e} B\'{a}rcenas }
        \address{
Centro  de  Ciencias Matem\'aticas. \\ Universidad  Nacional Aut\'onoma  de  M\'exico \\Ap. Postal  61-3 Xangari\\ Morelia, Michoac\'an.  M\'EXICO 58089}
         \email{barcenas@matmor.unam.mx}
         \urladdr{http://www.matmor.unam.mx/~barcenas}

\author{Mario Vel\'asquez}
\address{Centro  de  Ciencias Matem\'aticas. \\ Universidad  Nacional Aut\'onoma  de  M\'exico \\ Ap. Postal 61-3 Xangari \\Morelia, Michoac\'an.  M\'EXICO 58089}
 \email{mavelasquezm@gmail.com}
         \date{\today}


\usepackage{multicol}
\usepackage{hyperref}
\usepackage{pdfsync}
\usepackage{calc}
\usepackage{enumerate,amssymb}
\usepackage[arrow,curve,matrix,tips,2cell]{xy}
  \SelectTips{eu}{10} \UseTips
  \UseAllTwocells

\usepackage{pb-diagram}
\usepackage{verbatim}

\DeclareMathAlphabet\EuR{U}{eur}{m}{n}
\SetMathAlphabet\EuR{bold}{U}{eur}{b}{n}

\makeindex             

\theoremstyle{plain}
\newtheorem{theorem}{Theorem}[section]
\newtheorem{lemma}[theorem]{Lemma}

\newtheorem{corollary}[theorem]{Corollary}

\newtheorem*{theoremn}{Theorem}

\theoremstyle{definition}
\newtheorem{definition}[theorem]{Definition}

\newtheorem{remark}[theorem]{Remark}

{\catcode`@=11\global\let\c@equation=\c@theorem}

\newcommand{\comsquare}[8]                   
{\begin{CD}
#1 @>#2>> #3\\
@V{#4}VV @V{#5}VV\\
#6 @>#7>> #8
\end{CD}
}

\newcommand{\xycomsquare}[8]                   
{\xymatrix
{#1 \ar[r]^{#2} \ar[d]^{#4} &
#3 \ar[d]^{#5}  \\
#6\ar[r]^{#7} &
#8
}
}

\newcommand{\calfin}{\mathcal{FIN}}

\newcommand{\UU}{\mathcal{U}}
\newcommand{\HH}{\mathcal{H}}
\newcommand{\KK}{\mathcal{K}}

\newcommand{ \calr}{\mathcal{R}}

\newcommand{\IC}{{\mathbb C}}

\newcommand{\IN}{{\mathbb N}}

\newcommand{\IZ}{{\mathbb Z}}

\newcommand{\curs}{\EuR}

\newcommand{\CHAINCOMPLEXES}{\curs{CHCOM}}

\newcommand{\MODULES}{\curs{MODULES}}
\newcommand{\Or}{\curs{Or}}

\newcommand{\colim}{\operatorname{colim}}

\newcommand{\Spin}{\operatorname{Spin}}

\newcommand{\sldrei}{SL_{3}{\mathbb{Z}}}
\newcommand{\stein}{St_3{\mathbb{Z}}}

\newcommand{\ktheory}[4]{{}^{#1}K_{{#2}}^{#3}{(#4)}}
\newcommand{\so}[2]{u_{#1}^ {#2}}

\newcommand{\eub}[1]{\underline{E}#1}              
\newcommand{\Orfin}{\Or_\calfin(G)}

\newcommand{\higherlim}[3]{{\setbox1=\hbox{\rm lim}
        \setbox2=\hbox to \wd1{\leftarrowfill} \ht2=0pt \dp2=-1pt
        \mathop{\vtop{\baselineskip=5pt\box1\box2}}
        _{#1}}^{#2}#3}

\newcommand{\version}[1]                       
{\begin{center} last edited on #1\\
last compiled on \today \\
name of texfile: \jobname
\end{center}
}

\newcounter{commentcounter}

\begin{document}

  \begin{abstract}
In this  work, we  compare twisted  Equivariant  K-theory of $\sldrei$ with untwisted  equivariant  $K$-Theory  of a   central  extension $\stein$.  We  compute  all twisted  equivariant $K$-theory  groups  of $\sldrei$,  and compare  with  previous work  on the equivariant  $K$- Theory  of  $B\stein$ by  Tezuka  and  Yagita.

Using a universal  coefficient  theorem  by the  authors, the  computations explained  here give  the domain of Baum-Connes assembly maps  landing  on the  topological  $K$-theory of  twisted group $C^*$-algebras related to  $\sldrei$, for  which a  version  of  $KK$-Theoretic Duality studied  by  Echterhoff, Emerson and  Kim  is  verified.

  \end{abstract}
  \maketitle
  
\section{Introduction}  
In this  note, we   compare  versions  of  twisted  equivariant  $K$-theory with  respect  to a  discrete  group $G$, and  \emph{untwisted} equivariant  $K$-theory  of a universal  central  extension  of  $G$. 

Given a discrete  group $G$,  a  proper $G$-CW  complex  $X$ and a cohomology  class  $\alpha$  in  the  third  Borel  cohomology  group $H^3(X \times_G EG, \mathbb{Z})$,  twisted  equivariant  $K$  theory, denoted  by $\ktheory{\alpha}{G}{*}{X} $  was  defined  in \cite{BEJU2013}.  

 Specializing  to  the  classifying  space  $\eub{G}$ of  proper  actions    of  $G$ and  performing  the  Borel  construction $\eub{G}\times_G EG$  gives  a  model  for  $BG$ and  thus  all  twistings  agree with elements  in  the cohomology  groups $H^3(BG,\mathbb{Z})$.    
  
In the  case of  a  discrete  group $G$ (compare \cite{mooreextensions}, \cite{moorepadical}),  a  class  $\alpha\in H^3(BG, \mathbb{Z}) =H^2(BG, S^1)$ determines  a central extension
$$1\to  S^1 \to \widetilde{G}_\alpha \overset{p_\alpha}{\to} G\to  1.$$

The  space  $\eub{G}$  with the  $\widetilde{G}_\alpha$-action  given  by  precomposition  with $p_\alpha$ is  a  model  for the  classifying  space  of  proper  actions  of  $\widetilde{G}_\alpha$, denoted  by  $\eub{{\widetilde{G}_\alpha}}$.
We  compare the  abelian  groups 
$\ktheory{}{\widetilde{G}_\alpha}{*}{\eub{\widetilde{G}_\alpha}}$
and $\ktheory{\alpha}{G}{*}{\eub{G}}$. 

We pay  specific  attention  to  the  groups $\sldrei$ and  $\stein$, related  by a  a central extension  of  the  form     

$$1\to  \mathbb{Z}/2 \to \stein \to \sldrei\to 1.$$

The  integral cohomology  of both  groups $\stein$, and  $\sldrei$ has  been  extensively  studied in \cite{Soule(1978)}, where  also a  model  for  the  classifying  space  for  proper  actions $\eub{\sldrei}$ was  constructed. In  third  degree, the  cohomology  groups  are finitely  generated, $2$-torsion,  and  generated  by  classes  $u_1, u_2$  in the  case  of $\sldrei$  and a  single  class  $w_1$ in the  case  of  $\stein$. 

We  describe the  restriction of  the  classes $u_1$ and $u_2$  to  the  cohomology  of  finite  subgroups  of  $\sldrei$ in Section \ref{sectiontwists}, where  also  the  relation  to  the  generating  class $w_1$  is  stated.  We  follow  these  classes  to  their  restrictions  on finite  subgroups   of $\stein$, which  are  covers $2$ to  $1$  of  finite  subgroups  of  $\sldrei$. 

It  turns  out  that  the  torsion  class  $u_1+u_2$ represents  the  central  extension 
$$1\to \IZ/2\IZ \to  \stein \overset{p}{\to}  \sldrei\to  1, $$

and  its  restriction  to  finite  groups  $H\leq  G$ gives  a  model  for  Schur  covering  groups of  $H$: 
$$1 \to \IZ/2\IZ \to  p^-1(H) \to  H, $$
(However,  more  finite  subgroups  appear  in  $\stein$ that  are not  a  Schur  covering  group  for  any finite  group  of $\sldrei$).

Thus,  a  cocycle  representing  $u_1+u_2$  and  the  central  extension satisfy  the  hypotheses  of  the following   Theorem (\ref{theoremuntwistingcentral} )

\begin{theoremn}
Let $G$ be a discrete group and  let $\alpha\in Z^2(G;S^1)$ be a cocycle taking values in $\IZ/n\IZ$. Consider the extension associated to $\alpha$
$$\xymatrix{1\ar[r]&\IZ/n\IZ\ar[r]&G_\alpha\ar[r]^{\rho}&G\ar[r]&1.}$$

Denote  by  $\eub{G}$ a  model  for  the  classifying  space  of  proper  actions and  notice  that  the  action  of  $G_\alpha$ via   $\rho $ on $\eub{G}$  exhibits  the  later  space as  a  model  for  $\eub{G_\alpha}$.

Then, the  map $\rho$  gives  an  isomorphism  of  abelian  groups  between  the   Bredon  cohomology  groups of  $\eub{G}$  with  coefficients  in  the  $\alpha $-twisted  representation  ring   and  the  Bredon  cohomology  groups  of  $\eub{G_\alpha}$  with  coefficients  in  the $1$-central  group representation Bredon module ( defined  in  \ref{defkcentralbredonmodule}). In  symbols, 

$$H^*(\underbar{E}G;\calr_\alpha^G)\xrightarrow{\rho^*}H^*(\underbar{E}G_\alpha;\calr_1^{G_\alpha}).$$

\end{theoremn}

We  use  the (Bredon) cohomological  description  to  feed  a spectral  sequence  constructed  to  compute  twisted  equivariant  $K$-theory  which  was  constructed  in \cite{BAVE2013}. The  input  of the  Spectral  sequence  are  Bredon  Cohomology groups  with  coefficients  in  twisted  representations,  as briefly  introduced  in Section \ref{sectionbredon}. The  spectral  sequence  is  seen  to  collapse  at  the  $E_2$-term  and  the twisted  equivariant $K$-theory  groups  are  determined. 
 \begin{itemize}
 
\item   (Theorem \ref{theoremtwistedktheoryu1})
The  twisted  equivariant  $K$-theory  groups with respect  to  $u_1$  are  as  follows: 
$${ }^{u_1}K^0_{\sldrei}(\underbar{E}\sldrei)\cong\IZ^{\oplus13} \quad   { }^{u_1}K^1_{\sldrei}(\underbar{E}\sldrei)= 0.$$

\item (Theorem \ref{theoremtwistedktheoryu2}) The  twisted  equivariant $K$-theory  groups  with  respect  to  $u_2$  are  as  follows: 
$${ }^{u_2}K^0_{\sldrei}(\underbar{E}\sldrei)\cong\IZ^{\oplus7},\quad { }^{u_2}K^1_{\sldrei}(\underbar{E}\sldrei)= 0.$$

\item (Theorem \ref{theoremtwistedktheory}) The  twisted  equivariant  $K$-theory  groups  with  respect  to $u_1+u_2$ are  as  follows :
$${ }^{u_1+u_2}K^0_{\sldrei}(\underbar{E}\sldrei)\cong\IZ^{\oplus5}, \quad { }^{u_1+u_2}K^1_{\sldrei}(\underbar{E}\sldrei)\cong\IZ/2\IZ.$$
\end{itemize}

Using  the  Universal  Coefficient  Theorem for  Bredon  cohomology  with coefficients  in twisted  representations, Theorem 1.13 in \cite{BAVE2013},  the previous  groups  are  verified  to  be  isomorphic  to  some  equivariant  $K$-Homology  groups with  coefficients  defined in terms  of  Kasparov $KK$-Theory  groups in Section \ref{sectionapplications}, Theorem  \ref{theoremkhomology}. This  extends  and  generalizes  work  by  S\'anchez-Garc\'ia  in \cite{Sanchez-Garcia(2006SL)} in the  untwisted  setting. 

 A  version  of  the  Baum-Connes  Conjecture  with coefficients, \cite{echterhoffchabert} relates  these  groups  to  the  topological  $K$-theory  of  twisted  group $C^*$-algebras. We  see  that  the input  of  the Baum Connes  map  with  coefficients given  by  the  twistings  $u_1$, $u_2$ and  $u_1+ u_2$  satisfy a version  of  KK-theoretic  Duality  studied  in \cite{echterhoffemersonkim}, and  verified  in \cite{BAVE2013} for  the  twist  $u_1$.

The  result  is  interpreted  in  terms  of  twisted  equivariant  $K$- theory  of  the  classifying  space $B \sldrei$  using  results by  Tezuka  and  Yagita \cite{tezukayagita}, the  Atiyah-Segal  Completion Theorem 4.4 on page  611  of  \cite{lueckolivercompletion}.

 This  work  is organized  as follows: 
 In  Section 2,  we  introduce  Bredon  (Co)-homology,  focusing  on  coefficients  in twisted  representations. In  Section  3,  we  review  spectral  sequences  relating  Bredon  cohomology  groups  to  versions  of  twisted  equivariant  $K$-theory.  Section  4  deals  with  the  proof  of  Theorem \ref{theoremuntwistingcentral},  relating twisted  equivariant  $K$-theory  and  untwisted   $K$-theory  which  is  equivariant  with  respect  to  a  central  extension coding  the  twist.  
Section  5  describes  cohomological  information  determining  the  twists,   as  well  as  some  misunderstandings  in  the  literature  concerning  the universal  central  extension  of  $\sldrei$  and $\stein$,  see  \ref{remarkcentralstein}. Section  6  deals  with  the  computations  in  Bredon  cohomology. Finally,  Section  7  gives   interpretations  of  the  results   as  computations  of  twisted  equivariant $K$-homology  related to  versions with  coefficients  of  the  Baum-Connes  Conjecture,  as  well  as  computations of  the complex  $K$  theory  of  the  classifying  space  $B \stein$ by  Tezuka  and  Yagita.

\subsection*{Acknowledgements}  
The  first  author  thanks  the  support  of  a CONACYT Postdoctoral  fellowship. The  second author  thanks  the  support  of  a  UNAM Postdoctoral Fellowship.

The  first  author  thanks  Prof. Pierre  de la  Harpe  for  enlightening correspondence  related  to the  difference between  $\stein$  and    the  universal  central  extension  of  $\sldrei$.    

Both  authors  thank  an  anonymous   referee  for  making  crucial  suggestions   about   both  the  presentation  and  the  mathematical  content  of this  note,  particularly  the  suggestion  of the  material  in section 4,  which  helped  the  authors  to  identify  a  mistake  in  a  previous  version of  this  work.

\section{Bredon  (co)-homology }\label{sectionbredon}
We  recall  briefly some definitions  relevant  to  Bredon  homology and  cohomology, see \cite{valettemislin}  for  more  details. 
Let $G$ be a discrete group. A $G$-CW-complex is a CW-complex with a $G$-action permuting the cells and such that
if a cell is sent to itself, this is done by the identity map. We call the $G$-action proper if all cell stabilizers are finite
subgroups of $G$.
\begin{definition} A model for $\eub G$ is a proper $G$-CW-complex $X$ such that for any proper $G$-CW-complex $Y$ there is a
unique $G$-map $Y\rightarrow X$, up to $G$-homotopy equivalence.
\end{definition}
One can prove that a proper $G$-CW-complex $X$ is a model of $\eub G$ if and only if the subcomplex of fixed points $X^H$ is
contractible for each finite subgroup $H\subseteq G$.  It can be shown that classifying spaces for proper actions always exist.

Let $\Orfin$
be the orbit category of finite subgroups of $G$; a category with one object $G/H$ for each finite subgroup $H\subseteq G$ and where  morphisms are  given  by  $G$-equivariant maps.  There  exists  a  morphism $\phi:G/H\rightarrow G/K$ if and only if $H$ is conjugate in $G$  to a subgroup of $K$.

\begin{definition}[Bredon  chain complex]
Let  $X$  be  a  proper $G$-CW-complex.  The  contravariant   functor  $\underline{C}_{*}(X):\Orfin\to  \mathbb{Z}-\CHAINCOMPLEXES  $  assigns  to every  object  $G/H$     the  cellular $\mathbb{Z}$-chain  complex   of  the $H$-fixed point  complex    $ \underline{C}_{*}(X^ {H})\cong C_{*}({\rm  Map  }_{G}(G/H, X))$  with  respect  to  the  cellular  boundary  maps $\underline{\partial}_{*} $. 
\end{definition}

We  will  use  homological  algebra  to  define Bredon  cohomology. 

A  contravariant  coefficient  system  is  a  contravariant  functor $ M:\Orfin\to \mathbb{Z}-\MODULES$.    Given  a  contravariant  coefficient  system $M$, the  Bredon    cochain   module $C_G^n(X;M)$ is  defined  as the   abelian  group   of  natural  transformations   of  functors  defined  on  the  orbit  category $\underline{C}_{n}(X) \to  M$. In  symbols, 

$$C_G^n(X;M)={\rm Mor}_{{\rm Funct}(\Orfin,\mathbb{Z}-\MODULES)}(\underline{C}_n(X),M)$$

Given a  set  $\{e_{\lambda}\}$ of   orbit   representatives of  the n-cells of  the  $G$-CW  complex  $X$,  and isotropy  groups  $H_{\lambda}$ in $G$ of  the  cells  $e_{\lambda}$,   the  abelian  groups $C_G^n(X,M)$  satisfy:
 
 $$C_G^n(X,M)= \underset{\lambda}{\bigoplus }Hom_{\mathbb{Z}}(\mathbb{Z}[e_{\lambda}], M(G/H_{\lambda}))$$   
 with  one  summand  for  each  orbit representative  $e_\lambda$.
They  afford  a differential $\delta^n:C_G^n(X,M)\to C_G^{n+1}(X,M)$ determined  by  $\underline{\partial}_*$ and pullback maps $M(\phi):  M(G/H_\mu)\to M( G/ H_\lambda )$  for  morphisms  $\phi:G/H_\lambda \to  G/H_\mu$.

\begin{definition}[Bredon cohomology] 
Let  $M$   be  a  contravariant coefficient  system. The  Bredon  cohomology  groups   with  coefficients  in  $M$, denoted  by  $H^{*}_{G} (X,  M)$    are  the  cohomology   groups  of  the  cochain  complex  $\big (C_{G}^ *(X, M), \delta^* \big )$. 
\end{definition}
A covariant coefficient system is a covariant functor $N:\Orfin\rightarrow \IZ-\MODULES$. Let $N$ be a covariant coefficient system and $X$ be a proper $G$-CW-complex.  Dually  to  the  cohomological  situation, one can define the Bredon homology groups with coefficients in $N$. We denote these by $H_*^G(X,N)$.  Details can be found in pages  14-15  of \cite{valettemislin}.

\subsection*{Bredon  (co)-homology  with  coefficients  in twisted  representations. }
\begin {definition}
Let  $K$ be  a  finite  subgroup  in  the  discrete  group $G$.   Let $V$  be  a  complex  vector  space and $S^1$ be the unit circle in the complex numbers. Given  a   cocycle     $\alpha:K\times K\to S^{1}$  representing  a  class  in  $ H^{2}(BK,S^{1})\cong H^{3}(BK,\mathbb{Z})$, an $\alpha$-twisted  representation  is  a   function  to  the  general  linear group of  $V$, $P:K\to Gl(V)$  satisfying: 
$$P(e)={\rm 1}$$  
$$P(x)P(y)=\alpha(x,y)P(xy).$$
The Grothendieck group of isomorphism classes of $\alpha$-twisted representations is called the \textit{$\alpha$-twisted representation group} and it is denoted by $R_\alpha(K)$
\end{definition}
Two  $\alpha, \alpha{'}$-twisted  representations  are  isomorphic  if   the  cocycles $\alpha$, $\alpha{'}$  are  cohomologous  in  $H^{2}(BK,S^{1})$.  
\begin{definition}
Let  $H$  be  a  finite  group  and  let  $\alpha  \in  Z^2(H,  S^1)$  be  a  cocycle. Recall  that  the  $\alpha$-twisted  Complex  group  algebra $ \IC ^\alpha  H$  is  generated  as a  complex  vector  space   by  the  elements $\{h\mid  h\in H\}$. The  multiplication  is  given  by   the  following  formula on  representatives: 
$$h_1 h_2 = \alpha (h_1, h_2) h_1 h_2, $$

and  extended  $\IC$-linearly  to  define  a  complex  algebra structure  on $\IC^\alpha H$.  
\end{definition}

It  is a  consequence  of  Theorem 3.2 in  page  112 ,  Volume  2, part  1  of \cite{karpilovsky},  that  the $K_0$ group  of  the $\alpha$- twisted  complex group  algebra $\IC^\alpha H$  agrees  with  the $\alpha$-twisted representation group $R_\alpha(H)$.

 We  define a contravariant and a  covariant  coefficient  system  for  the  family  $\mathcal{F}_G= \calfin$  of  finite  subgroups    
 agreeing  on  objects  by using the $K_0$-group  of  the  twisted  group  algebra,  using  restriction  to  define  the  contravariant  functoriality,  and using induction  to  define the covariant  functoriality.

 \begin{definition}\label{definitiontwistedcoefficients}
 Let $G$  be  a  discrete  group and  let   $\alpha\in Z^ {2}(G,S^ {1})$ be  a  cocycle.  Let  $i:  H\to G$  be  an  inclusion  of  a  finite  subgroup $H$. 
 
 Define $\calr_\alpha$ on objects $G/H$ by
$$\calr_\alpha (G/H):= K_0(\IC^{i^{*}\alpha}(H))\cong R_{i^{*}(\alpha)}(H).$$

Let  $ \phi:G/H\rightarrow G/K$  be a $G$-equivariant map, we denote by $\calr{_{\alpha}}_?(\phi):R_{\alpha\mid}(H)\rightarrow R_{\alpha\mid}(K)$ the induction of  $\alpha$-twisted, representations for  the covariant  functor. For the contravariant functor, we denote by  $  \calr{_{\alpha}}^? (\phi):R_{\alpha\mid}(K)\rightarrow R_{\alpha\mid}(H)  $ the restriction  of $\alpha$-twisted representations. 
 \end{definition}

\begin{definition}\label{defbredoncohom}
Let $G$  be  a  discrete  group, let  $X$  be  a  proper $G$-CW complex, and  let  $\alpha\in Z^{2}(G, S^ {1})$  be  a  cocycle. The  $\alpha$-twisted Bredon  (co)-homology groups  of  $X$  are  the  Bredon  (co)-homology  groups  with  respect  to  the   functors  described  in Definition \ref{definitiontwistedcoefficients}. 
\end{definition}

\begin{remark}
Notice  the   role  of  the  family  of   finite  groups   in  definition \ref{defbredoncohom}. More  generally,  one  can  define  Bredon  (co)-homology  groups   for  a family  $\mathcal{F}$  of  subgroups which  contains  the  isotropy  groups  of  a  $G$-CW  complex  $X$,  and  a  functor $\mathcal{F}\to \IZ-\MODULES$.  Since  we  are  dealing  with  proper  actions on $G$-CW  complexes, we  can concentrate  on Bredon cohomology  for  the  family  of   finite  subgroups.  
\end{remark}

\section{ Spectral  sequences  for  Twisted  Equivariant  $K$-Theory. }\label{sectionspectral}

Twisted  equivariant  $K$-theory  for  proper  and  discrete  actions   has  been  defined  in a  variety  of  ways. For  a  torsion  cocycle  $\alpha \in Z^{2}(G,  S^ {1})$, it is possible to define it  in  terms  of  finite  dimensional, so  called  $\alpha$-twisted  vector  bundles, as  for example in  \cite{dwyer2008}.  This  is  not  possible  for  twistings  of  infinite order, and  the  general  approach of  \cite{BEJU2013} or  $C^*$-algebraic  methods  are needed. 

\begin{definition}\label{definitionktheory}
Let  $\alpha\in Z^ {2}(G,S^1)$  be a  normalized  torsion cocycle  of order  $n$   for  the  discrete  group  $G$,  with associated  central  extension  
$$0\to  \mathbb{Z}/{n} \to  G_{\alpha}\to  G.$$  An  \textit{$\alpha$-twisted  vector  bundle}  is  a  finite  dimensional $G_{\alpha}$-equivariant  complex vector bundle  such  that  $\mathbb{Z}/n$  acts  by  multiplication  with a  primitive $n$-th root  of  unity.  The  \textit{$\alpha$-twisted, $G$-equivariant  K-theory}  groups $^{ \alpha }{K}^{0}_{G}( X)$ are  defined  as   the  Grothendieck  groups  of  the  isomorphism classes of  $\alpha$-twisted  vector  bundles over  $X$.
  
\end{definition}
 Given  a proper  $G$-CW  complex $X$,  define $^{\alpha}K^{-n}_{G}(X)$  as  the  kernel  of  the  induced  map 

$$ ^{ \alpha }{K}^{0}_{G}( X\times S^{n})   \overset{{\rm incl}^{*}}{\to}       {  ^{ \alpha }{K}^{0}_{G}( X)}. $$
 
The  $\alpha$-twisted equivariant  $K$-theory   catches   information  relevant  to  the class  of  twistings coming from  the  torsion  part  of  the  group  cohomology  of  the  group, in the  sense  that  the $K$-groups are  zero  for  cocycles  representing non-torsion  classes. In contrast,  the  approach  discussed in \cite{BEJU2013} overcomes  this  difficulty.

As  noted in \cite{barcenasespinozauribevelasquez}, there is a spectral sequence connecting the $\alpha$-twisted Bredon cohomology and the $\alpha$-twisted
equivariant K-theory of finite proper $G$-CW complexes. When the twisting is given  by a  torsion  element  of  $H^3( BG, \IZ)$,  this spectral sequence is a special case of the Atiyah-Hirzebruch spectral sequence  for \emph{untwisted} $G$-cohomology theories constructed by Davis and L\"uck \cite{davisluck}. In particular, it  collapses  rationally.
\begin{theorem}[\cite{barcenasespinozauribevelasquez}]\label{spectral} Let $X$ be a finite proper $G$-CW complex for a discrete group $G$, and let $\alpha\in Z^2(G,S^1)$ be a
normalized torsion cocycle. Then there is a spectral sequence with

$$E_2^{p,q}=\begin{cases}
  H^p_G(X,\calr_{\alpha}^?)&\text{ if $q$ is even}\\
                     0 &\text{ if $q$ is odd}
\end{cases}$$
so that $E_\infty^{p,q}\Rightarrow { }^\alpha K^{p+q}_G(X)$.
\end{theorem}

\section{$S^1$-central extensions and torsion cocycles.}

\begin{definition}\label{rep defzklineal}
Let  $ 1\to  \IZ/n \IZ \to  \tilde{H} \to  H$  be  a  central  extension. Let  $k$ ba a natural number with $0\leq k\leq n$. Let  $V$  be  a a  complex  vector  space.  A  $k$-central  representation of  $\tilde{H}$  is  a  map $\tilde{H} \to Gl(V)$,  where  the  generator $t \in   \IZ /n \IZ$ acts  by  multiplication  by $e^{2\pi i k/n}$.

\begin{definition}\label{defkcentralbredonmodule}
The $k$-central  representation  group of  $\tilde{H}$, denoted  by  $R_k(\tilde{H})$, is the  Grothendieck  group  of  isomorphism  classes  of  $k$-central  representations  of  $\tilde{H}$. 
\end{definition}
The $k$-central representation group is a contravariant coefficient system. Given a discrete group $\tilde{G}$, we denote by $\calr_k^?$  the functor
\begin{align*}
\calr_k^?:\Or_\calfin(\tilde{G})&\rightarrow \IZ-\MODULES\\
\tilde{G}/\tilde{H}&\mapsto R_k(\tilde{H}). 
\end{align*}
\end{definition}

\begin{lemma}\label{lemmauntwisting}
Let $G$  be  a  discrete  group  and  let  $\alpha\in Z^2(G;S^1)$ be a torsion cocycle of order $n$.  Then, 

\begin{itemize}
\item There  exists  a  cocycle $\gamma$  with  values  on $\IZ/n\subset S^1$, which  is  cohomologous  to  $\alpha$. 

\item There  exists  a  central  extension  of  the  form
\begin{equation*}\xymatrix{1\ar[r]&\IZ/n\IZ\ar[r]&G_\alpha\ar[r]^{\rho}&G\ar[r]&1.}
\end{equation*}
With  the  property  that  for  each finite  group $H\leq G$, 
the  $1$-central representation  group $\calr_1(\rho^{-1}(H))$  is  isomorphic  to  the  $\alpha\mid_H$- twisted  representation  ring ${}^{\alpha}\calr(H)$ as an   abelian  group.

\item Moreover,  this  extends  to  a natural  transformation  of contravariant  functors
defined  over  the  orbit  category of  $G$,  

$$T: {}^{\alpha}\calr^?    \cong \calr^?_1\circ \rho   $$
which  consists  of  group  isomorphisms  on  each orbit. 
\end{itemize}

\end{lemma}

\begin{proof}

\begin{itemize}
\item Let $\alpha\in Z^2(G;S^1)$ be a torsion cocycle of order $n$. Then  $\alpha^n$ is cohomologous to the trivial cocycle, i.e there is a cochain $t\in C^1(G,S^1)$ with $\alpha^n=\delta t$. 

Define a cochain $u\in C^1(G, S^1)$ by $u(g)=(t(g))^{-\frac{1}{n}}$. The cocycle $\gamma=\alpha \cdot \delta u$ is again torsion of order $n$. The  cocycle $\gamma$ takes values in $\IZ/n\IZ$ and it is cohomologous to $\alpha$. 

\item  We use  the $\IZ/n\IZ$-valued cocycle  $\gamma$  to  define  a  group structure  on  the  set $G\times \IZ/n\IZ$,  and  obtain   a central extension of $G$ by $\IZ/n\IZ$,  denoted  by $G_\alpha$ (the  notation  being  justified by  the  fact  that  $\alpha$  is  cohomologous to  $\gamma$).

Let $\sigma$  be  the generator of $\IZ/n\IZ$ and $0\leq i\leq n-1$.
The multiplication on the  group $G_\alpha$ is given on elements $(g,\sigma^i)$  by 
$$(g,\sigma^j)\cdot(h,\sigma^i)=(gh,\alpha(g,h)\sigma^{j+i}),$$
thus  defining  a  central  extension 
\begin{equation*}\label{extensiontwisting}\xymatrix{1\ar[r]&\IZ/n\IZ\ar[r]&G_\alpha\ar[r]^{\rho}&G\ar[r]&1.}
\end{equation*}

\item  Let  $H$ be  a  finite  subgroup  of  $G$. Given  a  torsion cocycle $\alpha$,  consider  the   central  extension 
$$1\to \IZ/n\IZ\to  G_\alpha \overset{\rho}{\to} G\to  1,$$
which  was  constructed  in  the  previous  part. 

Let $\beta: H\to Gl(V)$  be an  $\alpha$-representation.  We  define  the 1-central  representation  $T(H)$  as  the   isomorphism  class  determined  by  the  map $\rho^{-1}(H)\to Gl(V)$  defined   on  elements  $(h, t) \in H\times \IZ/ n\IZ$ by $t\cdot (\beta(t,h))$,  where  we  consider $t\in  \IZ/n\IZ \subset S^ 1$ and $\cdot$  denotes  complex  multiplication. 

This  defines  a  group  homomorphism   
$$T: {}^{\alpha}\calr (G/H)    \to \calr_1( G_\alpha / \rho^{-1}(H) ). $$
An  inverse  to  the   homomorphism  $T$  is  given  by  assigning to the  $1$-central  representation  $\epsilon: \tilde{H}\to GL(V)$  the  projective  representation $\kappa : H\to GL(V)$  given  by  $\kappa(h)= \epsilon(h, 1)$. One  checks  that  this  is  a  $\gamma$-representation, where  $\gamma$  is  the cocycle with values  on $\IZ/n\IZ$ constructed  in the  first  part.

\item Let  $H$  and  $K$  be  finite  subgroups  of $G$. 
 The  map $\rho: G_\alpha \to  G$
defines  a  functor 
\begin{align*}
\Or_\calfin(G)&\xrightarrow{\rho^*}\Or_\calfin(G_\alpha)\\
G/H&\mapsto G_\alpha/\rho^{-1}(H)
\end{align*}
between  the  orbit  categories  with  respect  to  the family  of  finite  subgroups. 

We  will  analyze  the  behaviour of  the  functor  $T$  with  respect  to  restriction. 

Let  $\phi: G/H\to  G/K$  be  a  $G$-equivariant  map.  Recall  that  such  a  map  is  determined  up  to  $G$-conjugacy  by  an inclusion $H\to K$  of finite  subgroups of $G$.

Given an  $\alpha$-projective representation $ \beta:H\to GL(V)$, the  following  diagram  is  commutative  
$$\xymatrix{  K \ar[r] & H \ar[r]^{\beta} & GL(V) \\ \rho^{-1}(K) \ar[r] \ar[u]^{\rho}&  \rho^{-1}(H) \ar[u]^{\rho} \ar[ru]_{T(\beta)}&  },   $$
where  the  unlabelled  arrow  denote inclusions. 

Hence,  the  functor  $T$  is  compatible  with  restrictions  and  thus  defines a  natural  transformation  of  contravariant  functors over  the  orbit  category.

\end{itemize}

\end{proof}



\begin{theorem}\label{theoremuntwistingcentral}
Let $G$ be a discrete group and  let $\alpha\in Z^2(G;S^1)$ be a cocycle taking values in $\IZ/n\IZ$. Consider the extension associated to $\alpha$
$$\xymatrix{1\ar[r]&\IZ/n\IZ\ar[r]&G_\alpha\ar[r]^{\rho}&G\ar[r]&1.}$$

Denote  by  $\eub{G}$ a  model  for  the  classifying  space  of  proper  actions and  notice  that  the  action  of  $G_\alpha$ via   $\rho $ on $\eub{G}$  exhibits  the  later  space as  a  model  for  $\eub{G_\alpha}$.

Then, the  map $\rho$  gives  an  isomorphism  of  abelian  groups  between  the   Bredon  cohomology  groups of  $\eub{G}$  with  coefficients  in  the  $\alpha $-twisted  representation  ring   and  the  Bredon  cohomology  groups  of  $\eub{G_\alpha}$  with  coefficients  in  the $1$-central  group representation group. In  symbols, 

$$H^*(\underbar{E}G;\calr_\alpha^G)\xrightarrow{\rho^*}H^*(\underbar{E}G_\alpha;\calr_1^{G_\alpha}).$$

\end{theorem}
\begin{proof}
Fix a $G$-cellular structure of $\underbar{E}G$.  Associate  to each equivariant cell in $\underbar{E}G$ of the form $G/H\times D^n$ a cell in $\underbar{E}G_\alpha$ of the form $G_\alpha/\rho^{-1}(H)\times D^n$.  

Consider the cellular cochain complex of $\underbar{E}G.$ In degree $n$, it has the form $$C_G^n(\eub{G},M)= \underset{\lambda}{\bigoplus }Hom_{\mathbb{Z}}(\mathbb{Z}[e_{\lambda}], { }^{\alpha}\calr(G/H_{\lambda})).$$
  
From \ref{lemmauntwisting}, this term is isomorphic (via $\rho^*$) to 
$$C_G^n(\eub{G},M)= \underset{\lambda}{\bigoplus }Hom_{\mathbb{Z}}(\mathbb{Z}[e_{\lambda}], \calr_1(G_\alpha/\rho^{-1}(H_{\lambda})))$$  
and the isomorphism commutes with the cellular boundary, thus determining a  chain  isomorphism 
 
$$C^*(\underbar{E}G;{ }^{\alpha}\calr^?)\xrightarrow{\rho^*} {C}^*(\underbar{E}G_\alpha;\calr_1^{?}),$$
which  induces  an isomorphism in Bredon cohomology.
\end{proof}

\begin{corollary}
Let $G$ be a discrete group and  let $\alpha\in Z^2(G;S^1)$ be a cocycle taking values in $\IZ/n\IZ$. Consider the extension associated to $\alpha$
$$\xymatrix{1\ar[r]&\IZ/n\IZ\ar[r]&G_\alpha\ar[r]^{\rho}&G\ar[r]&1.}$$
Then,  there  exists  an   isomorphism  of abelian  groups
$$\ktheory{\alpha}{G}{*}{\eub{G}} \cong  \ktheory{}{G_\alpha}{*}{\eub{G_\alpha}},$$
between  the  $\alpha$-twisted, $G$-equivariant  $K$-Theory  and  the  untwisted  $G_\alpha$-equivariant  $K$-theory   of  the  classifying  spaces  for  proper  actions $\eub{G}= \eub{G_\alpha}$.

\end{corollary}

\begin{proof}
From Theorem \ref{theoremuntwistingcentral},  the  Bredon  cohomology  groups are  all  isomorphic. The  spectral sequence  \ref{spectral}  lets  us  conclude  the  desired  isomorphism.  

\end{proof}

\section{Twistings in $\sldrei$ and  $\stein$. } \label{sectiontwists}

\subsection*{The  cohomology  of  $\sldrei$}
We recall the analysis  of  the  cohomology  of  $\sldrei$ in \cite{BAVE2013}. Soul\'e proved in \cite{Soule(1978)} that the  integral  cohomology  of  $\sldrei$ only  consists  of  $2$ and  $3$-torsion.  The  3-primary  part  is  isomorphic  to  the  graded   algebra 
$$\mathbb{Z}[x_{1}, x_{2}]/\langle3x_{1},3x_{2}\rangle$$     
 with  both  generators  in  degree 4. 

The  two-primary  component  is  isomorphic  to  the   graded  algebra  
$$\mathbb{Z}[u_{1}, \ldots, u_{7}]$$
with    respective degrees 3, 3, 4, 4, 5, 6, 6, subject to the   relations 
$$2u_{1}=2u_{3}=4 u_{3}=4u_{4}=2u_{5}=2u_{6}=2u_{7}=0$$
$$\so{7}{} \so{1}{} = \so{7}{}\so{4}{}=\so{7}{}\so{5}{}=\so{7}{}\so{6}{}= \so{2}{} \so{5}{}= \so{2}{} \so{6}{}=0 $$
$$\so{7}{2}+\so{7}{}\so{2}{2}=\so{3}{}\so{4}{}+\so{1}{} \so{5}{}=\so{3}{} \so{6}{}+\so{3}{}\so{1}{2}=\so{3}{}\so{6}{}+ \so{5}{2}=0$$
$$\so{1}{} \so{6}{}+\so{4}{}\so{5}{}=\so{3}{0} \so{4}{2}+\so{6}{2}=\so{5}{} \so{6}{}+ \so{5}{} \so{1}{2}=0$$

The  twistings  in  equivariant  $K$-theory  are  given  by   classes  in  $H^{3}(\sldrei, \mathbb{Z})$ all of which are 2-torsion.  For  this  reason,  we  shall  restrict  to  the  two-primary  component (we indicate this with the subscript $(2)$)  in  integral  cohomology. In  order  to  have  a  local  description  of  these  classes,  we   describe   the cohomology  of  some  finite  subgroups  inside $\sldrei$.

The  finite  groups  of  $\sldrei$   include  $S_4$,  the  symmetric  group in  four  letters, $D_4$, the  dihedral  group of  order  8, 
the  dihedral  group  of  order 12, $D_6$,  as  well  as   the  group  of  order  two denoted  by  $C_2$.

  Theorem 4 in page 14 of \cite{Soule(1978)}, gives the following result:
For all $n\in\IN$ there exists an exact sequence of abelian groups
$$0\rightarrow H^n(\sldrei)_{(2)}\xrightarrow{\phi}H^n(S_4)_{(2)}\oplus H^n(S_4)_{(2)}\oplus H^n(S_4)_{(2)}\xrightarrow{\delta} H^n(D_4)\oplus H^n(C_2)\rightarrow 0$$
where $\phi$  is given by  restrictions (see \cite{Soule(1978)}[2.1(b), Cor.]) and $\delta$ by the system of embeddings

 $$
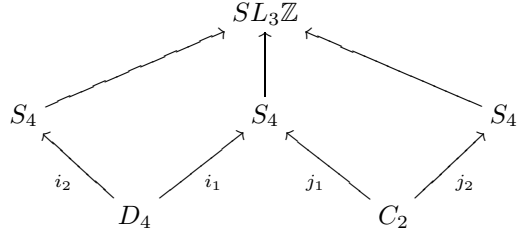
\xymatrix{&&\sldrei\\S_4\ar[urr] & & S_4\ar[u]& & S_4\ar[ull]\\
  & D_4\ar[ul]^{i_2}\ar[ur]_{i_1}& & C_2\ar[ul]^{j_1}\ar[ur]_{j_2}\\
    }.	 $$

 If $R$ is as in Proposition 4 in \cite{Soule(1978)}, 
 the image of the morphism $\phi: H^*(\sldrei)_{(2)}\rightarrow H^*(S_4)_{(2)}\oplus (i_1^*)^{-1}(R)$, is the set of elements $(y, z)$ such that 
 $j_2^*(y) = j_1^*(z).$ From the paper of Soul\'e, we know that $H^*(S_4)_{(2)} = \IZ[y_1, y_2,y_3]$, with 
 $2y_1 = 2y_z = 4y_3 = y_1^4+ y_2^2y_1 + y_3y_1^2= 0$, and, 
  $(i_1^*)^{-1}(R)=\IZ[z_1, z_2, z_3]$, with $2z_1 = 4z_2 = 2z_3 = z_3^2 + z_3z_1^2 = 0$. Furthermore $j_2^*(y_1) = t$, $j_2^*(y_2) = 0,$
$j_2^*(y_3) = t^2$, $j_1^*(z_1) = 0$, $j_1^*(z_2) = t^2$, and $j_1^*(z_3) = 0$. Then the elements $u_1 = y_2$, $u_2 = z_1$, $u_3 =y_1^2+ z_2$, $u_4 = y_1^2+ y_3$, $u_5 = y_1y_2$,  $u_6 = y_1y_3 + y_1^3$ and $u_7 = z_3$ generate $\phi(H^*(\sldrei)_{(2)})$.
  
 In $H^3(\ )$ the above discussion can be summarized in the following diagram
  


   \begin{equation}\label{diagram}\small
  
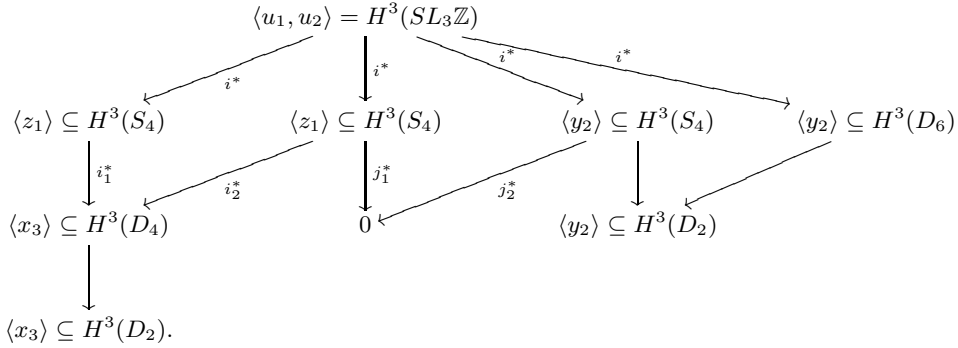
\xymatrix{&  \langle u_1,u_2\rangle=H^3(\sldrei)\ar[dl]^{i^*}\ar[d]^{i^*}\ar[dr]^{i^*}\ar[drr]^{i^*}\\
  \langle z_1\rangle\subseteq H^3(S_4)\ar[d]^{i_1^*} &\langle z_1\rangle\subseteq H^3(S_4)\ar[dl]^{i_2^*}\ar[d]^{j_1^*} & \langle y_2\rangle \subseteq H^3(S_4)\ar[dl]^{j_2^*}
  \ar[d]&\langle y_2\rangle \subseteq H^3(D_6)\ar[dl]\\
   \langle x_3\rangle \subseteq H^3(D_4)\ar[d]& 0 & \langle y_2\rangle \subseteq H^3(D_2)\\
   \langle x_3\rangle \subseteq H^3(D_2).
  }
  \end{equation}
 In  the  Following  section  we  will  give  explicit  generators   and   analyze  the  depicted embeddings  in  $\sldrei$.

\subsection*{The  cohomology  of  $\stein$. } 
 
The  following result  was  published  as Theorem 8, page  17 in  \cite{Soule(1978)}, see also  section 4  in \cite{tezukayagita}, page  92 for  a more precise account.

\begin{theorem}
\begin{itemize}
\item There  exists  a 3  torsion cohomology  class  $\xi\in H^4( \stein, \mathbb{Z})$ such  that, for  any  $\stein$-module $A$,  the  cup  pruduct  by  $\xi$ induces  an  isomorphism 
$$\cdot \cup \xi : H^k(\stein, A) \to H^{k+4}(\stein, A)$$
as  soon  as  $k>3$  and  $k>0$  when $A$  is  constant. 

\item The  ring  $H^*(\stein, \mathbb{Z})_2$  is  generated  by  elements $w_1$, $w_2$, $w_3$ with  respective  degrees  $3$, $4$, $4$,  submitted  to the  defining  relations $2w_1= 4w_2=16w_3={w_1}^2= w_1w_2=w_2w_3=0$. Hence $H^1(\stein, \IZ)= H^2(\stein, \IZ)=0$, $H^3(\stein, \IZ)=\IZ/2$, $H^4(\stein, \IZ)= \IZ/16\oplus \IZ/4\oplus \IZ/3$. 
   
\end{itemize}

\end{theorem}

The  cohomology   of  $\stein$  is seen  to  be  completely  determined  by   the  classes $w_1, w_2, w_3$, as  well  as  the  periodicity  class  $\xi$. The  classes  $w_i$ restrict  non-trivially  to  some  specific   generators  of  the  cohomology  of finite subgroups. We will  analyze  briefly  how  they  relate  to  the  generating classes  $u_1$, $u_2$ of  $H^3(\sldrei, \IZ)$. This  is  a  summary  of  the  discussion in  Lemma 9, and the  proof  of Theorem 8 in \cite{Soule(1978)}.

 The  group $St_3(\mathbb{R})$   fits  as  a  central  extension $1\to\IZ/2 \to  St_3(\mathbb{R}) \to  SL_3(\mathbb{R}) \to 1,  $ which  restricts  to a central  extension  of  lattices $1\to\IZ/2 \to \stein \to  \sldrei \to 1.$

 The  maximal  compact  subgroups  of   $St_{3}(\mathbb{R})$,  respectively $SL_{3}(\mathbb{R})$, are  $\Spin_3$,  respectively $SO(3)$. Hence, all  finite  subgroups  of $\stein$ are  contained  in $\Spin_3$,  which  is  homeomorphic  to the $3$-dimensional  sphere, thus  the  cohomology  of  all  finite  subgroups  in  $\stein$  is 4-periodic. This   is  the  origin  of  the periodicity class  $\xi$. 
 
 The  class  $w_1$ restricts  nontrivially  under  a  system  of  inclusions of  finite  groups $$\xymatrix{   S_4^* & &  S_4^* \\ & D_4^* \ar[ru] \ar[lu]  &  }, $$   which covers  the  inclusions $i_1, i_2:  D_4\rightrightarrows S_4$ in $\sldrei$.
 
 Thus, $u_1$  maps  to  $w_1$, and $u_2$ maps  to  the  trivial  class  under  the   map  induced  by the universal cover $\stein \to \sldrei$ in cohomology.

\begin{remark}\label{remarkcentralstein}[The  universal  central  extension of  $ \sldrei$  and $\stein$.]

In  the early  literature on  the  Steinberg  group (particularly  Steinberg' s  Yale  notes \cite{steinberg}),  there  is  an  unfortunate  identification  of  $\stein$  with the  universal  central  extension  of  $\sldrei$. This  mistake  has  been  repeated  in the  literature \cite{Soule(1978)}, 2.7  and  \cite{delaharpebekka}, Example  IV. 

Denote  by  $\widetilde{SL_n(\IZ)}$  the  universal  central  extension  of  $SL_n(\IZ)$. It  fits  in an  exact  sequence  
$$1\to H_2(SL_n(\IZ), \IZ) \to  \widetilde{SL_n(\IZ)}\to SL_n(\IZ).$$

 While  there  is  an  identification  of  $St_n(\IZ)$ with $\widetilde{SL_{n}(\IZ)}$  for  $n\geq  5$, Van der  Kallen \cite{vanderkallen}  computes  the  Schur  Multiplier $H_2(G, \IZ)$  for  $G$ $\sldrei$  and $SL_4(\IZ)$,  being  in  both  cases isomorphic  to Klein's Four  group $\IZ/ 2\IZ \oplus \IZ/ 2 \IZ$.  Thus, the  universal  central  extension  defining $\stein$, 
 $$ 1\to \IZ/ 2\IZ \to \stein \to  \sldrei,$$
 and the  one  defining  $\widetilde{\sldrei}$  
 $$1\to \IZ/ 2\IZ \oplus \IZ/2\IZ \to  \widetilde{\sldrei} \to \sldrei$$ 
 are  not  the  same. We  thank  Prof. Pierre De  la  Harpe  for  pointing  this  fact  to  us  on  personal  correspondence,  leading  to  the  correction  of a  mistake  in  a  previous  version  of  this  note. 
  
 \end{remark}

\section{Twisted K-theory of $\sldrei$}
We use the following notations: $\{1\}$ denotes the trivial group, $C_n$ the cyclic group
of $n$ elements, $D_n$ the dihedral group with $2n$ elements and $S_n$ the Symmetric group of permutations on $n$ objects.

There are  four twistings for $\sldrei$  up to cohomology, namely $0,u_1,u_2,u_1+u_2$, continuing the work started in \cite{BAVE2013}, we will calculate the twisted K-theory for the 
twistings $u_2$ and $u_1+u_2$.

From diagram \ref{diagram}, one can see that the class $u_2$ restricts nontrivially to two copies of $S_4$ corresponding to the stabilizer of the vertices $v_3$ and $v_5$. We recall the $\sldrei$-CW-complex structure of $\underbar{E}\sldrei$ as is given in \cite{Soule(1978)}. The labels $O$, $Q$, $M$, $N$, $P$ of the vertices refer to the Figure 2 of \cite{Sanchez-Garcia(2006SL)}, where also Soul\'es matrices
$g_1,\ldots, g_{14}$ are recalled.

\vspace{0.5cm}
\begin{center}
\begin{tabular}{ccccccccc}
\hline
vertices & & & & &2-cells\\
\hline
$v_1$ & $O$ & $g_2$, $g_3$ & $S_4$& &$t_1$ & $OQM$ & $g_2$ & $C_2$\\
$v_2$ & $Q$ & $ g_4$, $g_5$& $D_6$&&$t_2$ & $QM'N$ & $g_1$ & $\{1\}$\\
$v_3$ & $M$ &  $g_6$, $g_7$&  $S_4$& &$t_3$ & $MN'P$ & $g_{12},g_{14}$ & $C_2\times C_2$\\
$v_4$ & $N$ & $ g_6$, $ g_8$ & $ D_4$&& $t_4$ & $OQN'P$ & $g_5$ & $C_2$\\
$v_5$ & $P$ & $g_5$ , $g_9$ & $ S_4$ && $t_5$ & $OMM'P$ & $g_6$ & $C_2$\\
\hline
edges&&&&& 3-cells\\
\hline
$e_1$ & $OQ$ & $g_2$, $g_5$ &  $C_2\times C_2$ && $T_1$& $g_1$ & $\{1\}$\\
$e_2$ & $OM$ & $g_6, g_{10}$ & $D_3$\\
$e_3$ & $OP$ & $ g_6, g_5$ & $D_3$\\
$e_4$ & $QM$ & $g_2$ & $C_2$\\
$e_5$ & $QN'$ & $g_5$ & $C_2$\\
$e_6$ & $MN$ & $g_6, g_{11}$ & $C_2\times C_2$\\
$e_7$ & $M'P$ & $g_6, g_{12}$ & $ D_4$\\
$e_8$& $ N'P$ & $ g_5, g_{13}$ & $ D_4$\\

\hline

\end{tabular}
\end{center}
\vspace{0.5cm}

The first column is an enumeration of equivalence classes of cells; the second lists a representative of each class;
the third column gives generating elements for the stabilizer of the given representative; and the last one is the
isomorphism type of the stabilizer. The
generating elements referred to above are the same as in \cite{BAVE2013}.

\subsection*{The  twisting $u_1$.}
The  following  theorem  was  proved  in \cite{BAVE2013}: 

\begin{theorem}\label{theoremtwistedktheoryu1}
$${ }^{u_1}K^0_{\sldrei}(\underbar{E}\sldrei)\cong\IZ^{\oplus13},$$

$${ }^{u_1}K^1_{\sldrei}(\underbar{E}\sldrei)= 0.$$
\end{theorem}

\subsection*{The twisting $u_2$.}
In order to determine the twisted K-theory, we calculate  Bredon cohomology.

\subsubsection*{Determination of $\Phi_1$.}

In order to determine the morphism $\Phi_1$, we need to recall the projective character tables of the groups where $u_2$ restricts \textit{non trivially}.

Here we denote by $z$ the generator of the central copy of $\IZ_2$.
The linear  character  table  of a  Schur  covering   group ${S_{4}^*}$    is  obtained  on page  254, volume  3   of  \cite{karpilovsky} by considering  the  group   with   presentation

$$S_4^*=\langle h_1,h_2,h_3,z\mid h_i^2=(h_jh_{j+1})^3=(h_kh_l)^2=z, z^2=[z,h_i]=1\rangle$$ 
 $$1\leq i\leq 3,  j= 1, k\leq l-2$$

 and  the  central  extension 
$$1\to \langle z\rangle \to S_4^*\overset{f}{\to} S_4\to 1$$
given by $f(h_i)=g_i$, as  well as the  choice  of  representatives   of  regular conjugacy  classes  as  below.     
\[
\begin{array}{c|cccccccc}
 S_{4}^*& e   & z   & h_1 &  h_1h_3 &  h_1h_2  &  h_1h_2z  & h_1h_2h_3 & h_1h_2h_3z\\\hline
  \epsilon_{1}     &   1        &   1           &    1     &  1       &    1        &        1    &     1     &  1      \\
  \epsilon_{2}     &   1        &   1           &   -1     &  1       &    1        &        1    &     -1    &  -1     \\
  \epsilon_{3}     &   2        &   2           &    0     &  2       &   -1        &       -1    &     0     &   0     \\
  \epsilon_{4}     &   3        &   3           &    1     &  -1      &    0        &       0     &     -1    &   -1    \\
  \epsilon_{5}     &   3        &   3           &   -1     &  -1      &    0        &       0     &     1     &   1     \\\hline
  \epsilon_{6}     &   2        &  -2           &    0     &   0      &   1         &       -1    &  \sqrt{2} &-\sqrt{2}\\
  \epsilon_{7}     &   2        &  -2           &    0     &   0      &   1         &       -1    & -\sqrt{2} & \sqrt{2}\\
  \epsilon_{8}     &   4        &  -4           &    0     &    0     &   -1        &       1     &    0      &    0    \\

\end{array}
\]
 where  the   first  five   lines  are  characters   associated  to  $S_{4}$, and $\epsilon_{6}$  is  the  \emph{Spin } representation. 
 
We take the following presentation of Dihedral groups, $D_{n}=\langle g_{i},  g_{j}\rangle =\langle g_{i}, g_{j} \mid g_{i}^{2}=g_{j}^{2} =(g_{i}g_{j})^{n}=1\rangle$ 

The  dihedral group   of  order   six  has trivial 3  dimensional  integer  cohomology.  Thus  its  projective   representations  do  agree  with  the  linear  ones. The dihedral subgroups with $n$ even in $\sldrei$ are  $C_{2}\times  C_{2}=D_{2}$ and  $D_{4}$. 

The  following  is  the linear   character  table for  $D_{n}$:

\[\begin{array}{c|ccc}
D_{n}& \langle (g_{i}, g_{j})^{k}\rangle  & \langle g_{j}(g_{i} g_{j})^{k} \rangle \\ \hline
\xi_{1} & 1 & 1\\   
\xi_{2}  & 1 &  -1\\ 
\hat{\xi_{3}} &  -1^{k}& -1^{k}\\
 \hat{\xi_{4}} & -1^{k} & -1^{k+1}\\\hline
\phi_{p} & 2\cos(2 \pi p  k/n) &  0 
  \end{array}
\]

where $0\leq k\leq n-1$,  p  varies  from  1  to  $(n/2) -1$ ( $ n$  even)  or  $(n-1)/2$ ( $n$ odd) and  the  hat  denotes  a   representation which  only  appears  in the   case  n  even. 
The  group $D_2^*=\langle h_1, h_3, z\rangle$ is  isomorphic  to the eight elements quaternion group, and a linear  character  table  is  given  by   
\[
\begin{array}{c|ccccc}
 D_2^*& 1   & z   & \{h_1, h_1^{-1}\}&  \{h_3,h_3^{-1}\} &\{h_1h_3, (h_1h_3)^{-1}\} \\\hline
  \eta_{1}     &   1        &   1           &    1     &  1       &    1         \\
  \eta_{2}     &   1        &   1           &   1     &  -1       &    -1          \\
  \eta_{3}     &   1        &   1           &    -1     &  1       &   -1          \\
  \eta_{4}     &   1        &   1           &    -1     &  -1      &    1         \\\hline
  \eta_{5}     &   2        &   -2           &   0     &  0     &    0          \\
\end{array}
\]

A Schur cover of $D_4$ can be taken as $D_8=\langle a,x\mid a^4=x^2=e, xax^{-1}=a^{-1}\rangle$, whose character table is:

\[\begin{array}{c|ccccccc}
 D_8&e&a^4(=z)&a^2&a&a^3(=az)&x&ax\\\hline
\lambda_{1} & 1 & 1&1&1&1&1&1\\   
\lambda_{2}  & 1 & 1&1&1&1&-1&-1\\ 
\lambda_{3} & 1 & 1&1&-1&-1&1&-1\\
 \lambda_{4} & 1 & 1&1&-1&-1&-1&1\\
\lambda_{5} & 2 & 2&-2&0&0&0&0\\\hline
\lambda_{6}& 2 & -2&0&\sqrt{2}&-\sqrt{2}&0&0\\
\lambda_{7}& 2 & -2&0&-\sqrt{2}&\sqrt{2}&0&0\\
  \end{array}
\]

The relevant inclusions among stabilizers are the following. We give a conjugacy representative appearing in the corresponding character table when necessary.
\vspace{0.5cm}

$\begin{array}{ll}

stab(e_2)&\xrightarrow{i} stab(v_3)\\
\langle g_6,g_{10}\rangle&\rightarrow\langle g_6,g_7\rangle\\
g_6&\mapsto g_6\\
g_{10}&\mapsto g_7^{-1}g_6g_7\\
\end{array}$
$\begin{array}{ll}
stab(e_3)&\xrightarrow{i} stab(v_5)\\
\langle g_6,g_{5}\rangle&\rightarrow\langle g_5,g_9\rangle\\
 g_6&\mapsto g_9^{-1}g_5g_9\\
 g_5&\mapsto g_5\\
\end{array}$

\vspace{0.5cm}
$\begin{array}{ll}
stab(e_4)&\xrightarrow{i} stab(v_3)\\
\langle g_2\rangle&\rightarrow\langle g_6,g_7\rangle\\
g_2&\mapsto g_6g_7^2g_6g_7^{-1}\\
\end{array}$
$\begin{array}{ll}
stab(e_5)&\xrightarrow{i} stab(v_4)\\
\langle g_5\rangle&\rightarrow q_2^{-1}\cdot\langle g_6,g_7\rangle\cdot q_2^{-1}\\
g_5&\mapsto q_2^{-1}\cdot g_8\cdot q_2^{-1}\\
\end{array}$

\vspace{0.5cm}
$\begin{array}{ll}
stab(e_6)&\xrightarrow{i} stab(v_3)\\
\langle g_6,g_{11}\rangle&\rightarrow \langle g_6,g_7\rangle\\
g_6&\mapsto g_6\\
g_{11}&\mapsto g_7g_6g_7^{-1}g_6g_7\sim g_6\\
\end{array}$
$\begin{array}{ll}
stab(e_6)&\xrightarrow{i} stab(v_4)\\
\langle g_6,g_{11}\rangle&\rightarrow \langle g_6,g_8\rangle\\
g_6&\mapsto g_6=x\\
g_{11}&\mapsto (g_6g_8)^2=a^2\\
\end{array}$

$\begin{array}{ll}
stab(e_7)&\xrightarrow{i} stab(v_3)\\
\langle g_6,g_{12}\rangle&\rightarrow q_1^{-1}\cdot\langle g_6,g_7\rangle\cdot q_1\\
g_6(=x)&\mapsto q_1^{-1}\cdot(g_6g_7^2g_6)\cdot q_1\sim g_7^2\\
g_{12}(=ax)&\mapsto q_1^{-1}\cdot g_6\cdot q_1\sim g_6\\
\end{array}$
$\begin{array}{ll}
stab(e_7)&\xrightarrow{i} stab(v_5)\\
\langle g_6,g_{12}\rangle&\rightarrow \langle g_5,g_9\rangle\\
g_6&\mapsto g_9^{-1}g_5g_9\\
g_{12}&\mapsto g_9^2
\end{array}$
\vspace{0.5cm}

$\begin{array}{ll}
stab(e_8)&\xrightarrow{i} stab(v_4)\\
\langle g_5,g_{13}\rangle&\rightarrow q_2^{-1}\cdot\langle g_6,g_8\rangle\cdot q_2\\
g_5&\mapsto q_2^{-1}\cdot g_8\cdot q_2\\
g_{13}&\mapsto q_2^{-1}\cdot g_6\cdot q_2\\
\end{array}$
$\begin{array}{ll}
stab(e_8)&\xrightarrow{i} stab(v_5)\\
\langle g_5,g_{13}\rangle&\rightarrow\langle g_5,g_9\rangle\\
g_5&\mapsto g_5\\
g_{13}&\mapsto g_5g_9^2g_5\sim g_9^2\\
\end{array}$

\vspace{1cm}

Using the above inclusions and  elementary calculations with characters, particularly  the  rectification  procedure, Theorem 1.7 in \cite{BAVE2013},  we obtain a matrix of size $34\times33$ representing the morphism $\Phi_1$. The matrices representing the restrictions among stabilizers are the following. The signs corresponding to the coboundary map as in \cite{Sanchez-Garcia(2006SL)}.

$$\begin{array}{|c|cccc|ccc|ccc|}
\hline& & e_1& & & &e_2 && & e_3 &\\\hline
   &-1&0&0&0&-1&0&0&-1&0&0\\
   &0&-1&0&0&0&-1&0&0&-1&0\\
v_1&-1&-1&0&0&0&0&-1&0&0&-1\\
   &-1&0&-1&-1&-1&0&-1&-1&0&-1\\
   &0&-1&-1&-1&0&-1&-1&0&-1&-1\\\hline
\end{array}$$
 
$$\begin{array}{|c|cccc|cc|cc|}
\hline& & e_1& &  &e_4 & & e_5 &\\\hline
   &1&0&0&0&-1&0&-1&0\\
   &0&1&0&0&0&-1&0&-1\\
v_2&0&0&1&0&0&-1&-1&0\\
   &0&0&0&1&-1&0&0&-1\\
   &0&0&1&1&-1&-1&-1&-1\\
   &1&1&0&0&-1&-1&-1&-1\\\hline
\end{array}$$

$$\begin{array}{|c|ccc|cc|c|cc|}
\hline
&&  e_2&&e_4 &&e_6&e_7&\\\hline
   &0&0&1&1&1&-1&-1&0\\
  v_3 &0&0&1&1&1&-1&0&-1\\
   &1&1&1&2&2&-2&-1&-1\\
   \hline
\end{array}$$

$$\begin{array}{|c|cc|c|cc|}
\hline
&  e_5&&   e_6 &e_8&\\\hline
   
v_4 &1&1&1&-1&0\\
    &1&1&1&0&-1\\
   \hline
\end{array}$$

$$\begin{array}{|c|ccc|cc|cc|}
\hline
&&  e_3&&e_7 &&e_8&\\\hline

   &0&0&1&1&0&1&0\\
v_5&0&0&1&0&1&0&1\\
   &1&1&1&1&1&1&1\\
   \hline
\end{array}$$

The elementary divisors of the matrix representing 
the morphism $\phi$ is 1 repeated 12 times. The rank of this matrix is 12.

\subsubsection*{Determination of $\Phi_2$.}

The relevant inclusions among stabilizers are the following. We give a conjugacy representative appearing in the corresponding character table when necessary.

$\begin{array}{ll}
stab(t_3)&\xrightarrow{i}stab(e_6)\\
\langle g_{12},g_{14}\rangle&\rightarrow q_1^{-1}\cdot\langle g_6,g_7\rangle\cdot q_1\\
g_{12}&\mapsto q_1^{-1}\cdot g_6\cdot q_1\\
g_{14}&\mapsto q_1^{-1}\cdot g_{11}\cdot q_1\\
\end{array}$
$\begin{array}{ll}
stab(t_3)&\xrightarrow{i}stab(e_7)\\
\langle g_{12},g_{14}\rangle&\rightarrow \langle g_6,g_{12}\rangle\\
g_{12}&\mapsto g_{12}=ax\\
g_{14}&\mapsto g_{12}(g_6g_{12})^2=xa\\
\end{array}$

\vspace{0.5cm}
$\begin{array}{ll}
stab(t_3)&\xrightarrow{i}stab(e_8)\\
\langle g_{12},g_{14}\rangle&\rightarrow \langle g_5,g_{13}\rangle\\
g_{12}&\mapsto g_{13}(g_5g_{13})^2=xa\\
g_{14}&\mapsto (g_5g_{13})^2=a^2\\
\end{array}$

\vspace{0.5cm}
$\begin{array}{ll}
stab(t_4)&\xrightarrow{i}stab(e_8)\\
\langle g_5\rangle&\rightarrow \langle g_5,g_{13}\rangle\\
g_5&\mapsto g_5=x
\end{array} $

\vspace{0.5cm}
$\begin{array}{ll}
stab(t_5)&\xrightarrow{i}stab(e_6)\\
\langle g_6\rangle&\rightarrow (q_1q_2)^{-1}\cdot\langle g_6,g_{12}\rangle\cdot(q_1q_2)\\
g_6&\mapsto (q_1q_2)^{-1}\cdot(g_6g_11)\cdot(q_1q_2)\\
\end{array} $
$\begin{array}{ll}
stab(t_5)&\xrightarrow{i}stab(e_7)\\
\langle g_6\rangle&\rightarrow \langle g_6,g_{12}\rangle\\
g_6\mapsto &g_6\\
\end{array} $
\vspace{0.5cm}

Using the above inclusions, an elementary calculation yields a matrix of size  $33\times12$ representing the morphism $\Phi_2$. The matrices representing the restrictions among stabilizers  are the following.

$\begin{array}{|c|cc|cc|}
\hline
&t_1& &t_4&  \\\hline
   &1&0&1&0\\
e_1&0&1&0&1\\
   &0&1&1&0\\
   &1&0&0&1\\
\hline
\end{array}$
\hspace{1cm}
$\begin{array}{|c|cc|cc|}
\hline
&t_1& &t_5&  \\\hline
   &-1&0&1&0\\
e_2&0&-1&0&1\\
   &-1&-1&1&1\\
   \hline
\end{array}$
\hspace{1cm}
$\begin{array}{|c|cc|cc|}
\hline
&t_4& &t_5 & \\\hline
   &-1&0&-1&0\\
e_3&0&-1&0&-1\\
   &-1&-1&-1&-1\\
   \hline
\end{array}$
$$\begin{array}{|c|cc|c|}
\hline
&t_1& &t_2  \\\hline
e_4 &1&0&1\\
    &0&1&1\\
   \hline
\end{array}
\hspace{1cm}
\begin{array}{|c|c|cc|}
\hline
&t_2& t_4&  \\\hline
e_5&-1&1&0\\
   &-1&0&1\\
   \hline
\end{array}$$

$$\begin{array}{|c|c|c|cc|cc|}
\hline
&t_2& t_3&t_4&&t_5&  \\\hline

e_6 &2&1&0&0&0&0\\
   \hline
\end{array}$$

$$\begin{array}{|c|c|cc|}
\hline
&t_3&t_5&  \\\hline

e_7&-1&1&1\\
   &-1&1&1\\
   \hline
\end{array}
\hspace{1cm}
\begin{array}{|c|c|cc|}
\hline
&t_3&t_4&  \\\hline
  e_8&1&1&1\\
     &1&1&1\\
   \hline
\end{array}$$

The elementary divisors of the matrix representing 
the morphism $\phi$ is 1,repeated 7 times. The rank of this matrix is 7.

\subsubsection*{Determination of $\Phi_3$.}

The morphism $\Phi_3$  is  given by   blocks which are represented by the following matrices

$$\begin{array}{|l|l|}
\hline
&T_1 \\\hline
t_1&-1\\
   &-1\\\hline
   \end{array}
\hspace{1cm}   
     \begin{array}{|l|l|}
   \hline
&T_1 \\\hline
t_2&1\\\hline\end{array}
\hspace{1cm}
\begin{array}{|l|l|}
\hline
&T_1 \\\hline

  t_3 &-1\\
   &-2\\\hline\end{array}
   \hspace{1cm}
 \begin{array}{|l|l|}
   \hline
&T_1 \\\hline
t_4&1\\
   &1\\\hline\end{array}
   \hspace{1cm}
   \begin{array}{|l|l|}
    \hline
&T_1 \\\hline
t_5&-1\\
   &-1\\
   \hline
\end{array}$$

We have the Bredon  cochain complex

$$0\rightarrow\IZ^{\oplus19}\xrightarrow{\Phi_1^{u_2}}\IZ^{\oplus19}
\xrightarrow{\Phi_2^{u_2}}\IZ^{\oplus8}\xrightarrow{\Phi_3^{u_2}}\IZ\rightarrow 0.$$

Using the information concerning ranks  and  elementary divisors of  $\Phi_i^{u_2}$, we obtain

\begin{equation}\label{equationktheory}
H^p_{\sldrei}(\underbar{E}\sldrei,\calr_{u_2})= 0\text{, if } p>0,\quad  
H^0_{\sldrei}(\underbar{E}\sldrei,\calr_{u_2})\cong\IZ^{\oplus7}.
\end{equation}

Since the Bredon cohomology concentrates at low degree, the spectral sequence described in section \ref{spectral} collapses at level 2 and  we conclude
\begin{theorem}\label{theoremtwistedktheoryu2}
$${ }^{u_2}K^0_{\sldrei}(\underbar{E}\sldrei)\cong\IZ^{\oplus7},$$

$${ }^{u_2}K^1_{\sldrei}(\underbar{E}\sldrei)= 0.$$
\end{theorem}
\subsection*{The   twisting  $u_1+u_2$}
Now we continue with the calculation of ${ }^{u_1+u_2}K_{\sldrei}(\underbar{E}\sldrei)$. Notice that the classes $u_1$ and $u_2$ are disjoint, i.e they do not restrict simultaneously  to  a non-zero element in the  cohomology  of any subgroup of
$\sldrei$. This  observation and  diagram \ref{diagram} lead to the following.

\begin{remark}\label{u1u21}The matrix $\Phi_1^{u_1+u_2}$ corresponding to the twisting $u_1+u_2$ can be obtained as:

$$\begin{array}{|c|c|c|c|c|c|c|c|c|}\hline
 &e_1&e_2&e_3&e_4&e_5&e_6&e_7&e_8\\\hline
 v_1& u_1&u_1&u_1&0&0&0&0&0\\\hline
 v_2&u_1&0&0&u_1&u_1&0&0&0\\\hline
 v_3&0&u_2&0&u_2&0&u_2&u_2&0\\\hline
 v_4&0&0&0&0&u_2&u_2&0&u_2\\\hline
 v_5&0&0&u_2&0&0&0&u_2&u_2\\\hline
 \end{array}$$\\
 
 Where a $u_i$ in position $(j,k)$ means that we take the corresponding submatrix of $\Phi_1^{u_i}$ associated to the inclusion $stab(e_j)\rightarrow stab(v_k)$. 
\end{remark}

This matrix has size $14\times16$ and it has elementary divisors $(1,1,1,1,1,1,1,1,2)$  and its rank is 9.
\begin{remark}\label{u1u22}The matrix $\Phi_2^{u_1+u_2}$ corresponding to the twisting $u_1+u_2$ can be obtained as:

$$\begin{array}{|c|c|c|c|c|c|}\hline
 &t_1&t_2&t_3&t_4&t_5\\\hline
 e_1& u_1&0&0&u_1&0\\\hline
 e_2&u_0&0&0&0&u_0\\\hline
 e_3&0&0&0&u_0&u_0.\\\hline
 e_4&u_0&u_0&0&0&0\\\hline
 e_5&0&u_0&0&u_0&u_0\\\hline
 e_6&0&u_2&u_2&0&u_2\\\hline
 e_7&0&0&u_2&0&u_2\\\hline
 e_8&0&0&u_2&u_2&0\\\hline
 \end{array}$$\\
\end{remark}
Where an  $u_i$ in position $(j,k)$ means that we take the corresponding submatrix of $\Phi_1^{u_i}$ associated to the inclusion $stab(t_j)\rightarrow stab(e_k)$ ($u_0$ denotes the trivial cocycle). 

This matrix has size $16\times8$ and it has 1 as elementary divisor 7 times and its rank is 7.

Finally the matrix $\Phi_3^{u_1+u_2}$ corresponding to the twisting $u_1+u_2$ is the same as the matrix $\Phi_3^{u_2}$.

We have the following cochain complex

$$0\rightarrow\IZ^{\oplus14}\xrightarrow{\Phi_1^{u_1+u_2}}\IZ^{\oplus16}
\xrightarrow{\Phi_2^{u_1+u_2}}\IZ^{\oplus8}\xrightarrow{\Phi_3^{u_1+u_2}}
\IZ\rightarrow 0$$

Using the  data  of $\Phi_i^{u_1+u_2}$ concerning ranks and elementary divisors we obtain

\begin{align}\label{equationktheoryu2}
H^p_{\sldrei}(\underbar{E}\sldrei,\calr_{u_1+u_2})&= 0\text{, if } p>1,\\  
H^0_{\sldrei}(\underbar{E}\sldrei,\calr_{u_1+u_2})&\cong\IZ^{\oplus5}\\
H^1_{\sldrei}(\underbar{E}\sldrei,\calr_{u_1+u_2})&\cong\IZ/2\IZ.
\end{align}

Since the Bredon cohomology concentrates at low degree, the spectral sequence described in section \ref{spectral} collapses at level 2 and  we conclude
\begin{theorem}\label{theoremtwistedktheory}
$${ }^{u_1+u_2}K^0_{\sldrei}(\underbar{E}\sldrei)\cong\IZ^{\oplus5},$$

$${ }^{u_1+u_2}K^1_{\sldrei}(\underbar{E}\sldrei)= \IZ/2\IZ.$$
\end{theorem}

\section{Applications}\label{sectionapplications}
\subsection*{Twisted  equivariant  $K$-Homology and  the Baum-Connes  Conjecture}

The  Baum-Connes  Conjecture  \cite{baumconnes}, \cite{valettemislin}  predicts  for  a   discrete   group $G$ the  existence  of  an  isomorphism 
$$\mu_{i}: K^G_i(\eub{G})\to K_i(C_{r}^*(G))$$
 given  by  the  (analytical) assembly  map,  where  $C_r^ *(G)$  is  the   reduced  $C^*$-algebra of  the  group  $G$.
 
More  generally,  given any  $G$-$C^*$-Algebra,  the  Baum-Connes   conjecture  with  coefficients   predicts an isomorphism   given by an assembly  map   
$$\mu_{i}: K^G_i(\eub{G}, A)\to K_i( A\rtimes G )).$$
Where $ K^G_i(\eub{G}, A)$ is  defined  in  terms  of  equivariant  and  bivariant  $KK$-groups ,

$$K_*^G(\eub{G},A)= \underset{{\rm G-compact}X\subset \eub{G}}{\colim} KK_{*}(C_{0}(X), A)$$

 and  $  A\rtimes G$ denotes  the  crossed  product $C^ *$-algebra, $X\subset \eub{G}$  is  a cocompact  subcomplex. 
See \cite{echterhoffchabert},  \cite{echterhoff}  for  more  details.

\begin{definition}
Let $G$  be  a  discrete  group. Given a  cocycle  $\omega \in Z^2(G,S^1)$,  an  $\omega$-representation  on a  Hilbert  space  $\HH$  is  a   map  $V: G\to  \UU(\HH)$ 
satisfying $V(s)V(t)= \omega(s,t)V(st)$.
\end{definition}
 
Consider the  quotient  map $\UU(\HH)\to  PU(\HH)= \UU(\HH)/S^1$.  Recall that  the group $PU(\HH)$ is the outer automorphism group $Out(\mathcal{K})$ of  the  $C^*$-algebra  of  compact  operators on $\HH$, denoted  by   $\KK$. 
 The  cocycle $\omega$ defines  in this  way an  action  of  $G$ on  $\KK$.     This  algebra   is  denoted by  $\KK_\omega$.

Let  $G$ be  a  discrete group  with a  finite  model  for  $\eub{G}$. Let  $\omega\in Z^2(G, S^1)$  be  a  cocycle  and  assume  that  the  bredon  cohomology  groups  $H^*(\eub{G},{\mathcal{R}}^{-\omega})$ are   concentrated  in  degrees  0 and 1.

Then,  the  Universal  Coefficient  Theorem  for  Bredon  cohomology, Theorem  1.13  in \cite{BAVE2013} identifies   the  Bredon  homology  groups  $H_*^{\sldrei}(\eub{\sldrei}, \mathcal{R}^\alpha)$  with  the  Bredon  cohomology  groups  $H^*_{\sldrei}(\eub{\sldrei}, \mathcal{R}^\alpha)$. 
By inspecting  the  Bredon  cohomology groups computed  in \ref{equationktheoryu2},  \ref{equationktheory},   the  hypotheses  of  Corollary  7.3  in \cite{BAVE2013} are  satisfied for  the  twistings  $u_2$ and $u_1+u_2$. This  gives  a  duality  isomorphism 
$${}^\omega K_G^*(\eub{G})\to K^G_*(\eub{G}, \mathcal{K}_{-\omega}).$$

Similar  forms  of  Poincar\'e duality for  proper and  twisted actions  have  been  studied  by  Echterhoff, Emerson  and Kim  in  \cite{echterhoffemersonkim} (Theorem 3.1)  under assumptions  concerning  the  Baum-Connes  conjecture, particularly the  validity  of  the  Dirac-Dual-Dirac  Method for the  group $G$.

\begin{theorem}\label{theoremkhomology}
The  equivariant  $K$-homology  groups with  coefficients  in the $G$-$C^*$ algebra $\mathcal{K}_\omega$  are  given  as  follows: 
\begin{itemize}
\item $$K_0^{\sldrei}(\underbar{E}\sldrei, \mathcal{K}_{u_1})\cong\IZ^{\oplus13},$$

$$K_1^{\sldrei}(\underbar{E}\sldrei, \mathcal{K}_{u_1})= 0.$$
\item $$K_0^{\sldrei}(\underbar{E}\sldrei, \mathcal{K}_{u_2})\cong\IZ^{\oplus7},$$

$$K_1^{\sldrei}(\underbar{E}\sldrei, \mathcal{K}_{u_2})= 0.$$
\item $$K_0^{\sldrei}(\underbar{E}\sldrei, \mathcal{K}_{u_1+u_2})\cong\IZ^{\oplus5},$$

$$K_1^{\sldrei}(\underbar{E}\sldrei, \mathcal{K}_{u_1+u_2})\cong \IZ/2\IZ.$$
\end{itemize}
\end{theorem}

\subsection*{Relation  to  the work  of   Tezuka and Yagita}


In the  case  of  finite  order  twists given  by  cocycles $\alpha \in Z^2(G, S^1)$, the  finite dimensional, $\alpha$-twisted vector  bundle model  of  twisted equivariant $K$-theory  is  related  to  untwisted equivariant  $K$-theory  groups  in a  way  we  will  describe  below.

Recall  that  given a normalized   torsion cocycle $\alpha$,  there  exists  a  central  extension 
$$1\to  \IZ/n \to {G}_\alpha \to G\to 1$$

Let  $X$ be  a  $G$-connected  G-CW  complex.  The  $\alpha$-Twisted  $K$-theory  groups  are  seen   to  agree  with  the  abelian group  of  ${G}_\alpha$-equivariant,  complex  vector  bundles for  which the  generator  of  $\IZ/n$  acts  by  complex  multiplication  by  $e^{2\pi i n}$.  There  is  a  splitting 
\begin{equation} \label{splitting}
\ktheory{}{{G}_\alpha}{0}{X}\cong \underset{V\in {\rm Irr}(\IZ/n)}{\bigoplus}\ktheory{}{{G}_\alpha}{0}{X, V},
\end{equation}

where $\ktheory{}{{G}_\alpha}{0}{X, V}$ is  the  subgroup  of ${G}_\alpha$-equivariant, complex  vector  bundles for  which  the  action  of  the  central $\IZ/n$  on each fiber restricts  to  the   irreducible representation $V$, and  the definition  is  extended  to  other  degrees  using the  remarks following  definition \ref{definitionktheory}. 



Given a  discrete  group $G$ and a  normalized  torsion cocycle $\alpha$, Theorem 3.4  in \cite{dwyer2008}  proves  that  the  groups $\ktheory{\alpha}{G}{*}{X}$ extend  to a  $\IZ/2$-graded  equivariant  cohomology  theory on the  category  of  finite,   proper  $G$-CW  pairs. This theory   restricts  to  equivariant  K-theory \cite{lueckolivercompletion} in the  case  of  a  trivial  cocycle. The  groups  $\ktheory{\alpha}{G}{*}{X}$ have a  natural  graded $\ktheory{}{G}{*}{X}$-module  structure.

The  multiplicative  structure  on the  graded  ring $\ktheory{}{G}{*}{\eub{G}}$ is  well  known. Recall  the definition  of the  augmentation  ideal
$$I_{G}= \ker(\ktheory{}{G}{0}{\eub{G}}\overset{i^*}{\to} \ktheory{}{G}{0}{\eub{G}_0} \to \ktheory{}{\{e\}}{0}{\eub{G}_0} , $$
where  $\eub{G}_0 \to \eub{G}$  denotes  the  inclusion  of  the $0th$-skeleton  and  the  last  map  is   the  restriction map  associated  to the trivial  group $\{ e\}\subset G$. 

The following  result  is a  generalization  of the  Atiyah-Segal completion Theorem  and it  is  proved in  part  b)  of   4.4 in  \cite{lueckolivercompletion}, page  611.

\begin{theorem}\label{theorem atiyahsegal}
Let  $\eub{G}$  be  the  classifying  space  for  proper  actions. 
\begin{itemize}
\item if  $\eub{G}$  has  the  homotopy  type  of  a  finite  $G$-CW complex, then there  is an isomorphism 
$$K^*(B G)\cong \ktheory{}{G}{*}{\eub{G}}_{\hat{I}_{G}}  ,$$

where  the  right  hand  side  denotes  the completion with respect  to the  ideal $I_G$. 

\end{itemize}

\end{theorem}
Specializing  to  the  case  of  $\stein$,  the topological $K$-theory  ring  $K^*(B\stein)$  is  known  after  computations  by   Tezuka  and  Yagita  using  Brown Peterson spectra and  its  Conner-Floyd  isomorphism, Corollary 4.7 in page 93  of  page \cite{tezukayagita},  which  we  recall:  
\begin{theorem} \label{tezukayagita}
Localized at  the  prime 2,  the  topological  $K$-theory  of  $B\stein$  is  given as  follows: 
\begin{itemize}
\item $K^0(B\stein)= \IZ_{\hat{2}}^6  \oplus \IZ_{(2)}$
\item $K^1( B\stein)= \IZ_{\hat{2}}$, 
\end{itemize}
where $\IZ_{(2)}$ is  the  localization at  $2$, and $\IZ_{\hat{2}}$ denotes  the  2-adical  completion  of  the  integers. 
\end{theorem}

Putting  toghether   Theorems \ref{theoremtwistedktheory},  \ref{theorem atiyahsegal}, and \ref{tezukayagita} one  obtains: 
\begin{corollary}
The  completion  of  the equivariant  $K$-theory  groups $\ktheory{}{G}{*}{\eub{\stein}}$  computed  in \ref{theoremtwistedktheory} with  respect  to  the  augmentation  ideal $I_{\stein}$ is  given as follows: 

\begin{itemize}
\item $ \ktheory{}{G}{0}{\eub{\stein}}_{\hat{I}_{\stein}} = \IZ_{\hat{2}}^6  \oplus \IZ_{(2)}$
\item $ \ktheory{}{G}{1}{\eub{\stein}}_{\hat{I}_{\stein}}    = \IZ_{\hat{2}}$, 
\end{itemize}
\end{corollary}

\bibliographystyle{alpha}
\bibliography{twistedrefrito}
\end{document}